\documentclass[10pt,a4paper]{article}

\usepackage[T1]{fontenc}
\usepackage[bitstream-charter]{mathdesign}

\usepackage[latin1]{inputenc}							
\usepackage{amsmath}	
\usepackage{latexsym,xypic}								
\usepackage{geometry}
\usepackage{longtable}
\usepackage{array}
\usepackage{fancyvrb,shortvrb}

\usepackage{xcolor}
\definecolor{bl}{rgb}{0.0,0.2,0.6}

\usepackage{sectsty}
\usepackage[compact]{titlesec}
\allsectionsfont{\color{bl}\scshape\selectfont}

\makeatletter							
\def\printtitle{
    {\color{bl} \centering \huge \sc \textbf{\@title}\par}}		
\makeatother							

\title{Some Remarks on Groups with Action on Itself
\vspace*{10pt}}

\makeatletter							
\def\printauthor{
    {\color{bl} \centering \large \sc \textbf{\@author}\par}}				
\makeatother							

\author{	Ahmet Faruk Aslan , Alper Odaba\c{s} and  Enver $\ddot{{\rm O}}$nder
Uslu 	\vspace*{10pt}}

\usepackage{fancyhdr}
\usepackage{lastpage}	
\usepackage[runin]{abstract}			
\abslabeldelim{\quad}						%
\newtheorem{theorem}{Theorem}
\newtheorem{definition}[theorem]{Definition}
\newtheorem{example}[theorem]{Example}

\newtheorem{corollary}[theorem]{Corollary}

\newtheorem{remark}[theorem]{Remark}

\newtheorem{proposition}[theorem]{Proposition}
\newtheorem{condition}{Condition}
\newenvironment{proof}[1][Proof]{\noindent\textbf{#1.} }{\ \rule{0.5em}{0.5em}}

\begin{document}
\printtitle

\printauthor

\bigskip

\textbf{Address:} Department of Mathematics and Computer Sciences, Osmangazi University, Art and
Science Faculty, Eskisehir, Turkey.

\bigskip

\textbf{e-mail addresses:} \texttt{afaslan@ogu.edu.tr, aodabas@ogu.edu.tr, enveruslu@ogu.edu.tr}

\bigskip

\textbf{Abstract:} We implement GAP functions about groups with action on itself and investigate some basic properties of small groups with action on itself of order $<32$.

\bigskip

\textbf{Keywords:} GAP, center, central series, ideal, nilpotency.%

\bigskip


\section{Introduction}

The notion of group with action on itself ( or shortly group with action )
was introduced by T. Datuashvili in \cite{tamar1, tamar2} for solving the
problem stated by J.-L. Loday in \cite{loday, loday1}, to define algebraic
objects called \textquotedblleft coquecigrues\textquotedblright\ which would
have an analogous role for Leibniz algebras as groups have for Lie algebras.
In \cite{tamar1}, the notion of central series of groups with action was
defined and given analogue of Witt's construction \cite{ikion} for such
objects. For this a condition was found and constructed a subcategory of the
category of groups with action whose objects satisfy this condition which
was called Condition 1.

It was interesting to investigate what kind of properties do small groups
with action on itself have. In standard GAP \cite{gap} library, there is no
function about the groups with action on itself. For this first we added the
function \texttt{AllGwAOnGroup(G)}, to obtain all group with action from a
small group $G$. Then we added some functions to get ideals, center and
isomorphism family of a given group with action $GA$. Finally, we added the
functions \texttt{NilpotencyClassOfGwA(GA)}, \texttt{IsSingular(GA)},
\texttt{IsGwAC1(GA)} which check if the group with action $GA$ is nilpotent,
singular or satisfies Condition 1 or not. Additionally, to prove our main
result Theorem \ref{main}, we constructed an enumeration table of group with
action obtained from small groups with order $<32$.

\bigskip

\textbf{Acknowledgement : }We would like to thank T. Datuashvili for
valuable comments and suggestions while her visit to Eskisehir Osmangazi
University supported by TÜB\.{I}TAK grand 2221 konuk veya akademik izinli
bilim insan\i\ destekleme program\i .

\section{\textbf{Preliminaries}}

In this section, we will recall some definitions and properties of self
acting groups from \cite{tamar1}, \cite{tamar2}. Let $G$ be a group which
acts on itself from the right side; i.e we have a map $\varepsilon :G\times
G\rightarrow G$ with%
\begin{equation*}
\begin{tabular}{l}
$\varepsilon (g,g^{\prime }+g^{\prime \prime })=\varepsilon (\varepsilon
(g,g^{\prime }),g^{\prime \prime }),$ \\
$\varepsilon (g,0)=g,$ \\
$\varepsilon (g^{\prime }+g^{\prime \prime },g)=\varepsilon (g^{\prime
},g)+\varepsilon (g^{\prime \prime },g),$ \\
$\varepsilon (0,g)=0,$%
\end{tabular}%
\end{equation*}%
for $g,g^{\prime },g^{\prime \prime }\in G$. Denote $\varepsilon (g,h)=g^{h}$%
, for $g,h\in G$. The pair $G^{\bullet }=(G,\varepsilon )$ is called a group
with action. We denote the group operation additively, nevertheless the
group is not abelian in general.

If $(G^{\prime },\varepsilon ^{\prime })$ is another group with action. A
homomorphism $(G,\varepsilon )\longrightarrow (G^{\prime },\varepsilon
^{\prime })$ is a group homomorphism $\varphi :G\longrightarrow G^{\prime }$
for which the diagram%
\begin{equation*}
\xymatrix@R=50pt@C=50pt{ G \times G  \ar[r]^{\varepsilon} \ar[d]_{(\varphi,\varphi)} & G \ar[d]^{\varphi} \\
G' \times G' \ar[r]^{\varepsilon'} & G'  }
\end{equation*}%
commutes. In other words, we have%
\begin{equation*}
\varphi (g^{h})=\varphi (g)^{\varphi (h)},\text{ \ }g,h\in G.
\end{equation*}

If we consider an action as a group homomorphism $G\overset{v}{%
\longrightarrow }AutG$, then a homomorphism between two groups with action
means the commutativity of the diagram%
\begin{equation*}
\xymatrix@R=25pt@C=65pt{ G \ar[r]^-{v} \ar[dd]_{\varphi} & AutG \subset Hom(G,G) \ar[d]^{Hom(G,\varphi)} \\
 & Hom(G,G') \\
 G' \ar[r]^-{v'} & AutG' \subset Hom(G',G') \ar[u]_{Hom(\varphi,G')}  }
\end{equation*}%
so that $\varphi \cdot (v(h))=v^{\prime }(\varphi (h))\cdot \varphi $, $h\in
G$.

We shall denote the category of groups with action by $\mathfrak{Gr}%
^{\bullet }$. Let $\mathfrak{Ab}^{\bullet }$ denote the category of abelian
groups with action; here we mean the action within $\mathfrak{Gr}^{\bullet }$%
. We have the functors%
\begin{equation*}
\xymatrix@R=40pt@C=40pt{ \mathfrak{Ab}^{\bullet} \ar@<0.5ex>[r]^{E} & \mathfrak{Gr}^{\bullet} \ar@<3.2ex>[r]^{Q_{1}} \ar@<0.5ex>[l]^{A} \ar@<-1.2ex>[r]^{Q_{2}}  & \mathfrak{Gr} \ar@<3.2ex>[l]_{C} \ar@<-1.2ex>[l]_{T} }
\end{equation*}%
where $Q_{1}(G),$ for $G\in \mathfrak{Gr}^{\bullet }$, is the greatest
quotient group of $G$ which makes the action trivial; $Q_{2}(G)$ is a
quotient of $G$ by the equivalence relation generated by the relation $%
g^{h}\sim -h+g+h$, $g,h\in G;$ $A$ is the abelianization functor, thus $%
A(G)=G/(G,G),$ where $(G,G)$ is the ideal of $G$ generated by the commutator
normal subgroup of $G$. $A(G)$ has the induced operation of action on
itself. Each group can be considered as a group with the trivial action or
with the action by conjugation, these give functors $T$ and $C$,
respectively. Every object of $\mathfrak{Ab}^{\bullet }$ can be considered
as an object of $\mathfrak{Gr}^{\bullet }$; this functor is denoted by $E$.
It is easy to see that the functors $Q_{1},Q_{2}$ and $A$ are left adjoints
to the functors $T,C$ and $E$ respectively. Also we have the forgetful
functor $\mathfrak{U}:\mathfrak{Gr}^{\bullet }\longrightarrow \mathfrak{Gr}$
which takes a group with action to its underlying group. In the sequel, the
groups with actions will be denoted by the letters $G^{\bullet
}=(G,\varepsilon _{G}),H^{\bullet }=(H,\varepsilon _{H}),...$

Now we will present some definitions and propositions from \cite{tamar1}
without proofs.

\begin{definition}
\label{ideal}Let $G^{\bullet }\in \mathfrak{Gr}^{\bullet }$. Let $A$ be
nonempty subset of $G$. If the conditions

\begin{enumerate}
\item[\textbf{i.}] $A$ is a normal subgroup of $G$ as a group,

\item[\textbf{ii.}] $a^{g}\in A$, for $a\in A$,$g\in G,$

\item[\textbf{iii.}] $-g+g^{a}\in A$, for $a\in A$ and $g\in G$,
\end{enumerate}

\noindent satisfied, then $A^{\bullet }$ is called an ideal of $G^{\bullet }$%
.
\end{definition}

Thus an ideal of $G^{\bullet }$ is a subobject of $G^{\bullet }$ in $%
\mathfrak{Gr}^{\bullet }$. It is clear that $G^{\bullet }$ itself and the
trivial subobject of $G^{\bullet }$ are ideals of $G^{\bullet }$. An
intersection of any system of ideals of $G^{\bullet }$ is an ideal, and
therefore we conclude that there exists the ideal generated by a system of
elements of $G^{\bullet }$.

\begin{proposition}
Let $G^{\bullet }\in \mathfrak{Gr}^{\bullet }$ and $A^{\bullet }$ be an
ideal of $G^{\bullet }$. For $a_{1},a_{2}\in A$, $g_{1},g_{2}\in G$ we have
\begin{equation*}
(a_{1}+g_{1})^{a_{2}+g_{2}}\in g_{1}^{g_{2}}+A\text{.}
\end{equation*}
\end{proposition}

Let $A^{\bullet }$ and $B^{\bullet }$ be subobjects of $G^{\bullet }$.
Denote by $\{A,B\}$ the subobject of $G^{\bullet }$ generated by $A^{\bullet
}$ and $B^{\bullet }$, and let $A+B$ denote the subset of $G$%
\begin{equation*}
A+B=\left\{ a+b:a\in A,b\in B\right\} .
\end{equation*}

\begin{proposition}
If $A^{\bullet }$ is an ideal of $G^{\bullet }$ and $B^{\bullet }$ is a
subobject of $G^{\bullet }$, then%
\begin{equation*}
\{A,B\}=A+B
\end{equation*}
\end{proposition}

\begin{proposition}
If $A^{\bullet }$ and $B^{\bullet }$ are ideals of $G^{\bullet }$, then $A+B$
is also an ideal.
\end{proposition}

\begin{proposition}
If $A^{\bullet }$ is an ideal of $G^{\bullet }$, then the quotient group $%
G/A $ with the induced action on itself is an object of $\mathfrak{Gr}%
^{\bullet } $.
\end{proposition}

In what follows, for $G^{\bullet }\in \mathfrak{Gr}^{\bullet }$ and $%
g,g^{\prime }\in G$, $\left[ g,g^{\prime }\right] $ will indicate the
element $-g+g^{g^{\prime }}$ of $G$ and $(g,g^{\prime })$ the commutator $%
-g-g^{\prime }+g+g^{\prime }$.

\begin{definition}
\label{comm}Let $A^{\bullet }$ and $B^{\bullet }$ be subobjects of $%
G^{\bullet }$. A commutator $\left[ A,B\right] $ of $A^{\bullet }$ and $%
B^{\bullet }$ in $G^{\bullet }$ is the ideal of $\{A,B\}$ generated by the
elements%
\begin{equation*}
\{\left[ a,b\right] ,\left[ b,a\right] ,(a,b):a\in A,b\in B\}
\end{equation*}
\end{definition}

\begin{definition}
\label{lower}The (lower) central series%
\begin{equation*}
G^{\bullet }=G_{1}^{\bullet }\supset G_{2}^{\bullet }\supset \cdots \supset
G_{n}^{\bullet }\supset G_{n+1}^{\bullet }\supset \cdots
\end{equation*}%
of the object $G^{\bullet }$ is defined inductively by%
\begin{equation*}
G_{n}^{\bullet }=\left[ G_{1},G_{n-1}\right] +\left[ G_{2},G_{n-2}\right]
+\cdots \left[ G_{n-1},G_{1}\right] .
\end{equation*}
\end{definition}

\begin{proposition}
For each $n\geq 1$, $G_{n+1}^{\bullet }$ is an ideal of $G_{n}^{\bullet }$.
\end{proposition}

\begin{condition}
For each $x,y,z\in G$,
\begin{equation*}
x-x^{(z^{x})}+x^{y+z^{x}}-x+x^{z}-x^{z+y^{z}}=0\text{.}
\end{equation*}
\end{condition}

In \cite{tamar1}, category of groups (Abelian groups) with action $\mathfrak{%
Gr}^{c}(\mathfrak{Ab}^{c})$ satisfying this condition and category of
Lie-Leibniz algebras $\mathfrak{LL}$ were defined. Then proved that the
analogue of Witt's construction defines a functor $LL:\mathfrak{Gr}%
^{c}\longrightarrow \mathfrak{LL}$ which gives rise to Leibniz algebras
(introduced in \cite{loday} ) over the ring of integers by compositions $%
\mathfrak{Gr}^{c}\overset{A}{\longrightarrow }\mathfrak{Ab}^{c}\overset{L}{%
\longrightarrow }\mathfrak{L}$eibniz, $\mathfrak{Gr}^{c}\overset{LL}{%
\longrightarrow }\mathfrak{LL}\overset{S_{2}}{\longrightarrow }\mathfrak{L}$%
eibniz. The details of the functors can be found in \cite{tamar1}.

\begin{example}
\cite{tamar1} Let $G$ be a group. Consider $G^{\bullet }$ as a group with \
action with the (right) action by conjugation. Then $G^{\bullet }$ satisfies
Condition 1.
\end{example}

\begin{example}
\cite{tamar1} Each group with the trivial action satisfies Condition 1.
\end{example}

\begin{example}
\cite{tamar1} Let $G^{\bullet }$ be the abelian group of integers $\mathbb{Z}%
^{\bullet }$, which acts on itself in the following way: $x^{y}=($-$1)^{y}x$%
. We have $[x,y]$ $=0$ for $y$ even, $[x,y]=$-$2x$ for $y$ odd and $%
G_{n}=2^{n-1}\mathbb{Z}$. It is easy to see that $\mathbb{Z}^{\bullet }$
satisfies Condition 1.
\end{example}

\begin{example}
\label{exsin}Consider the group $C_{2}\times C_{2}$ (Klein four). We
diagramized the multiplication table of $C_{2}\times C_{2}$ by
\begin{equation*}
\begin{tabular}{c||cccc}
$+$ & $e$ & $a$ & $b$ & $ab$ \\ \hline\hline
$e$ & $e$ & $a$ & $b$ & $ab$ \\
$a$ & $a$ & $e$ & $ab$ & $b$ \\
$b$ & $b$ & $ab$ & $e$ & $a$ \\
$ab$ & $ab$ & $b$ & $a$ & $e$%
\end{tabular}%
\end{equation*}%
We have ten groups with action obtained from $C_{2}\times C_{2}$. Four of
them satisfy Condition 1. Whose which will be denoted by $(C_{2}\times
C_{2},\varepsilon _{i})$, $i=1,2,3,4$. Table of the actions $\varepsilon
_{i} $ are as follows;
\begin{equation*}
\begin{tabular}{cc}
$%
\begin{tabular}{c||cccc}
$\varepsilon _{1}$ & $e$ & $a$ & $b$ & $ab$ \\ \hline\hline
$e$ & $e$ & $a$ & $b$ & $ab$ \\
$a$ & $e$ & $a$ & $b$ & $ab$ \\
$b$ & $e$ & $a$ & $ab$ & $b$ \\
$ab$ & $e$ & $a$ & $ab$ & $b$%
\end{tabular}%
$ & $%
\begin{tabular}{c||cccc}
$\varepsilon _{2}$ & $e$ & $a$ & $b$ & $ab$ \\ \hline\hline
$e$ & $e$ & $a$ & $b$ & $ab$ \\
$a$ & $e$ & $a$ & $b$ & $ab$ \\
$b$ & $e$ & $a$ & $b$ & $ab$ \\
$ab$ & $e$ & $a$ & $b$ & $ab$%
\end{tabular}%
$%
\end{tabular}%
\end{equation*}%
\begin{equation*}
\begin{tabular}{cc}
$%
\begin{tabular}{c||cccc}
$\varepsilon _{3}$ & $e$ & $a$ & $b$ & $ab$ \\ \hline\hline
$e$ & $e$ & $a$ & $b$ & $ab$ \\
$a$ & $e$ & $ab$ & $b$ & $a$ \\
$b$ & $e$ & $a$ & $b$ & $ab$ \\
$ab$ & $e$ & $ab$ & $b$ & $a$%
\end{tabular}%
$ & $%
\begin{tabular}{c||cccc}
$\varepsilon _{4}$ & $e$ & $a$ & $b$ & $ab$ \\ \hline\hline
$e$ & $e$ & $a$ & $b$ & $ab$ \\
$a$ & $e$ & $b$ & $a$ & $ab$ \\
$b$ & $e$ & $b$ & $a$ & $ab$ \\
$ab$ & $e$ & $a$ & $b$ & $ab$%
\end{tabular}%
$%
\end{tabular}%
\end{equation*}%
where any object $a_{ij}$ in the tables shows the right action of $a_{i}$ on
$a_{j}$.
\end{example}

\section{Center and Nilpotency Of A Group With Action}

In this section first we introduce the center of a group with action and
prove that our definition coincides with generalized definition of centers
in algebraic categories introduced by Hug in \cite{hug1}. Then we introduce
the notion of nilpotency (class) of a group with action and give the main
theorem which states that the nilpotent group with action obtained from
group with order $<32$ have nilpotency class $1$ or $2$.

\begin{definition}
\label{merkez}Let $G^{\bullet }\in \mathfrak{Gr}^{\bullet }$. Then
\begin{equation*}
\{g\in G:g+h=h+g,\text{ }g=g^{h},\text{ }h=h^{g},\text{ for all }h\in G\}
\end{equation*}%
is called the center of $G^{\bullet }$ and denoted by $Z(G^{\bullet })$.
\end{definition}

\begin{proposition}
\label{idea_teo}Let $G^{\bullet }\in \mathfrak{Gr}^{\bullet }$. Then $%
Z(G^{\bullet })$ is an ideal of $G^{\bullet }$.
\end{proposition}

\begin{proof}
Let $g,g^{\prime }\in Z(G^{\bullet })$. Since
\begin{equation*}
\begin{tabular}{cll}
$(g-g^{\prime })+h$ & $=$ & $g-g\prime +h$ \\
& $=$ & $g-(-h+g^{\prime })$ \\
& $=$ & $g-(g^{\prime }-h)$ \\
& $=$ & $g+h-g^{\prime }$ \\
& $=$ & $h+(h-g\prime ),$ \\
\multicolumn{1}{r}{} & \multicolumn{1}{c}{} &  \\
\multicolumn{1}{r}{$g-g^{\prime }$} & \multicolumn{1}{c}{$=$} & $%
g^{h}-g^{\prime h}$ \\
\multicolumn{1}{r}{} & \multicolumn{1}{c}{$=$} & $(g-g^{\prime })^{h},$ \\
&  &  \\
\multicolumn{1}{r}{$h^{(g-g^{\prime })}$} & \multicolumn{1}{c}{$=$} & $%
h^{-(g^{\prime }-g)}$ \\
& \multicolumn{1}{c}{$=$} & $(-h)^{g^{\prime }-g}$ \\
& $=$ & $-(h^{g^{\prime }-g})$ \\
& $=$ & $((h^{g^{\prime }})^{-g})$ \\
& $=$ & $-(h^{-g})$ \\
& $=$ & $h^{g}$ \\
& $=$ & $h,$%
\end{tabular}%
\end{equation*}%
for all $h\in G^{\bullet }$, we have $Z(G^{\bullet })$ is a subobject of $%
G^{\bullet }$. On the other hand, we have $h+g-h=g$, for all $g\in
Z(G^{\bullet })$, $h\in G^{\bullet }$ which means $Z(G^{\bullet })$ is a
normal subgroup of $G^{\bullet }$ as an additive group. Since $g^{h}=g$ and $%
-h+h^{g}=0$, we have $g^{h}\in Z(G^{\bullet })$ and $-h+h^{g}\in
Z(G^{\bullet })$ for all $g\in Z(G^{\bullet }),h\in G^{\bullet }$.
Consequently, $Z(G^{\bullet })$ is an ideal of $G^{\bullet }$.
\end{proof}

\begin{definition}
\cite{hug1} Two coterminal morphisms $\beta _{1}:B_{1}\longrightarrow A$ and
$\beta _{1}:B_{2}\longrightarrow A$ are said to commute if there exists a
morphism%
\begin{equation*}
\beta _{1}\circ \beta _{2}:B_{1}\times B_{2}\longrightarrow A
\end{equation*}%
making the diagram
\begin{equation*}
\xymatrix@R=45pt@C=45pt{
  B_{1} \ar[dr]_-{\beta_{1}} \ar[r]^-{\Gamma_{1}} & B_{1} \times B_{2} \ar[d]^-{\beta_{1} \circ \beta_{2}} & B_{2} \ar[l]_-{\Gamma_{2}} \ar[dl]^-{\beta_{2}} \\
   & A  &  }
\end{equation*}%
commutative, where $\Gamma _{i}$ for $i=1,2$ as usual denote the morphisms
of the direct product. In particular, the morphism $\alpha :B\longrightarrow
A$ said to be central if identity morphism on $A$ commutes with $\alpha $,
i.e., if it makes the diagram%
\begin{equation*}
\xymatrix@R=45pt@C=45pt{
  A \ar@{=}[dr]_-{1_{A}} \ar[r]^-{} & A \times B \ar[d]^-{} & B \ar[l]_-{} \ar[dl]^-{\alpha} \\
   & A  &  }
\end{equation*}%
commutative. In addition, if we have a monomorphism $\alpha
:B\longrightarrow A$, then it is said that $B$ is a central subobject of $A$.
\end{definition}

\begin{definition}
\label{mono}\cite{hug1} The center of an object is defined as the maximal
central subobject, relative to the order relation that exists on the set of
monomorphisms.
\end{definition}

\begin{proposition}
\label{subobject}Let $G^{\bullet }$ $\in \mathfrak{Gr}^{\bullet }$. Then $%
Z(G^{\bullet })$ is the maximal central subobject of $G^{\bullet }$.
\end{proposition}

\begin{proof}
Let $H^{\bullet }$ be central subobject of $G^{\bullet }$. Then there exist
a monomorphism $\alpha :H^{\bullet }\longrightarrow G^{\bullet }$ and an
homomorphism $\beta :G^{\bullet }\times H^{\bullet }\longrightarrow
G^{\bullet }$ which makes diagram%
\begin{equation*}
\xymatrix@R=45pt@C=45pt{
  G^{\bullet } \ar@{=}[dr]_-{} \ar[r]^-{} & G^{\bullet }  \times H^{\bullet } \ar[d]^-{\beta} & H^{\bullet } \ar[l]_-{} \ar[dl]^-{\alpha} \\
   & G^{\bullet }  &  }
\end{equation*}
commutative. So we have $\beta (g,0)=g$, for all $g\in G$ and $\beta
(0,h)=\alpha (h)$, for all $h\in H$.

Consequently, we have $\beta (g,h)=g+\alpha (h)$, for all $g\in G$, $h\in H$
which makes $\alpha $ a group with action homomorphism.

Indeed,%
\begin{equation*}
\begin{tabular}{ccl}
$\alpha (h+h^{\prime })$ & $=$ & $\beta (0,h+h^{\prime })$ \\
& $=$ & $\beta ((0,h)+(0,h^{\prime }))$ \\
& $=$ & $\beta (0,h)+\beta (0,h^{\prime })$ \\
& $=$ & $\alpha (h)+\alpha (h^{\prime })$%
\end{tabular}%
\end{equation*}%
and similarly, we have $\alpha \circ \varepsilon _{H}=\varepsilon _{G}\circ
(\alpha ,\alpha )$ as required.

Now, we will show that $\alpha (H^{\bullet })\subseteq Z(G^{\bullet })$.

Since
\begin{equation*}
\begin{array}{rcl}
\alpha (h)+g & = & \beta (0,h)+\beta (g,0) \\
& = & \beta (g,h) \\
& = & \beta ((g,0)+(0,h)) \\
& = & \beta (g,0)+\beta (0,h) \\
& = & g+\alpha (h), \\
&  &  \\
g^{\alpha (h)} & = & (\beta (g,0))^{\beta (0,h)} \\
& = & \beta ((g,0)^{(0,h)}) \\
& = & \beta (g^{0},0^{h}) \\
& = & \beta (g,0) \\
& = & g%
\end{array}%
\end{equation*}%
and similarly $(\alpha (h))^{g}=\alpha (h)$, for all $g\in G$, $h\in H$, we
have $\alpha (H^{\bullet })\subseteq Z(G^{\bullet })$, which means that $%
Z(G^{\bullet })$ is the maximal central subobject, as required.
\end{proof}

\begin{corollary}
Definition \ref{merkez} coincides with Hug's definition given in \cite{hug1}%
, i.e. $Z(G^{\bullet })$ is a maximal central subobject of $G^{\bullet }$ in
$\mathfrak{Gr}^{\bullet }.$
\end{corollary}

\begin{proof}
Follows from Definitions \ref{merkez}, \ref{mono} and Proposition \ref%
{subobject}.
\end{proof}

Let $G^{\bullet }\in \mathfrak{Gr}^{\bullet }$.\ Obviously $Z(G^{\bullet })$
is a normal subgroup of $Z(G)$. In addition, for any group $H$, we have the
group with action $H^{\bullet }$ whose underlying group is $H$ and the
action is defined by conjugation. In this case, $Z(H)$ coincides with $%
Z(H^{\bullet })$.

The category $\mathfrak{Gr}^{\bullet }$ is not a category of \ interest
since the binary operation (i.e. the action) is not distributive.
Nevertheless, we define the singular objects in $\mathfrak{Gr}^{\bullet }$
in analogues way as it is in \cite{Orz}.

\begin{definition}
We say a group with action is singular if it coincides with its center.
\end{definition}

\begin{example}
The Klein four group $C_{2}\times C_{2}$ with the action $\varepsilon _{1}$
defined in Example \ref{exsin} is singular.
\end{example}

A group with actions obtained from an abelian group is not a singular group
with action, in general. The following GAP session gives two groups with
action obtained from cyclic group $C_{4}$. One of them is singular and the
other is not.

\begin{Verbatim}[frame=single, fontsize=\small, commandchars=\\\{\}]
\textcolor{blue}{gap> K := SmallGroup(4,1);; StructureDescription(G);}
"C4"
\textcolor{blue}{gap> IsAbelian(K);}
true
\textcolor{blue}{gap> aKA := AllGwAOnGroup(K);;}
\textcolor{blue}{gap> List(aKA, i -> IsAbelianGwA(K));}
[ true, false ]
\end{Verbatim}

\begin{proposition}
\label{teo1}A group with action $G^{\bullet }$ is singular iff
\begin{equation*}
\left[ G^{\bullet },G^{\bullet }\right] =0
\end{equation*}
\end{proposition}

\begin{proof}
Direct checking.
\end{proof}

\begin{definition}
Let $G^{\bullet }$ be a group with action. $G^{\bullet }$ is called
nilpotent if $G_{n}^{\bullet }=0$ for some positive integer $n$ in the lower
central series of $G^{\bullet }.$ The less $(n-1)$ satisfying $%
G_{n}^{\bullet }=0,$ is called the nilpotency class of $G^{\bullet }$.
\end{definition}

\begin{corollary}
Let $G^{\bullet }\in \mathfrak{Gr}^{\bullet }$. If $G^{\bullet }$ is
singular then $G^{\bullet }$ is nilpotent with nilpotency class $1$.
\end{corollary}

\begin{proof}
Follows from proposition \ref{teo1}.
\end{proof}

\begin{proposition}
Let $G^{\bullet }$ be a nilpotent group with action. If $G^{\bullet }$ has
nilpotency class $1$ or $2$, then $G^{\bullet }$ satisfies Condition $1$.
\end{proposition}

\begin{proof}
As it is proved in \cite{tamar1}, Condition $1$ can be expressed as
\begin{equation*}
\lbrack x^{y},[y;z]]=[[x,y],z^{x}]+[[x,z],y^{z}],
\end{equation*}%
for all $x,$ $y,$ $z\in G^{\bullet },$ called as Condition $1^{\prime },$
from which the required results obtained.
\end{proof}

\begin{theorem}
\label{main} Let $G^{\bullet }\in \mathfrak{Gr}^{\bullet }$ whose underlying
group has order $<32$. If $G^{\bullet }$ is nilpotent, then its nilpotency
class is $1$ or $2.$
\end{theorem}

\begin{proof}
Follows from the table given in section $4$.
\end{proof}

\begin{example}
Not only the nilpotent groups with action, also some other groups with
action which are not nilpotent satisfy Condition $1$. The group with action $%
G^{\bullet }$ obtained from
\begin{equation*}
Q_{8}=\left\langle a,b,c\right\rangle =\left\{ e,a,b,c,ab,ac,bc,abc\right\}
\end{equation*}%
quaternion group of order $8$, whose action table is
\begin{equation*}
\begin{tabular}{c||cccccccc}
$\varepsilon $ & $e$ & $a$ & $b$ & $c$ & $ab$ & $ac$ & $bc$ & $abc$ \\
\hline\hline
$e$ & $e$ & $a$ & $b$ & $c$ & $ab$ & $ac$ & $bc$ & $abc$ \\
$a$ & $e$ & $a$ & $b$ & $c$ & $ab$ & $ac$ & $bc$ & $abc$ \\
$b$ & $e$ & $ac$ & $abc$ & $c$ & $bc$ & $a$ & $ab$ & $b$ \\
$c$ & $e$ & $a$ & $b$ & $c$ & $ab$ & $ac$ & $bc$ & $abc$ \\
$ab$ & $e$ & $ac$ & $abc$ & $c$ & $bc$ & $a$ & $ab$ & $b$ \\
$ac$ & $e$ & $a$ & $b$ & $c$ & $ab$ & $ac$ & $bc$ & $abc$ \\
$bc$ & $e$ & $ac$ & $abc$ & $c$ & $bc$ & $a$ & $ab$ & $b$ \\
$abc$ & $e$ & $ac$ & $abc$ & $c$ & $bc$ & $a$ & $ab$ & $b$%
\end{tabular}%
\end{equation*}%
satisfies the\ Condition $1$.
\end{example}

\begin{remark}
There are $52$ groups with action obtained from $Q_{8}$. $36$ of them are
not nilpotent and $6$ of these $36$ satisfy Condition $1.$
\end{remark}


\section{GAP Implementations Of Groups With Action}

There is no function for presenting of a group with action in standard GAP
library. First, we add the function \texttt{AllGwAOnGroup(G)}, whose output
is all groups with action obtained from a given group $G,$ by constructing
all homomorphisms from $G$ to its automorphism group. For example, if we
take $G=S_{3}$, we have the following;

\begin{Verbatim}[frame=single, fontsize=\small, commandchars=\\\{\}]
\textcolor{blue}{gap> G := SmallGroup(6,1);}
<pc group of size 6 with 2 generators>
\textcolor{blue}{gap> StructureDescription(G);}
"S3"
\textcolor{blue}{gap> aGA := AllGwAOnGroup(G);}
[ GroupWithAction [ Group( [ f1, f2 ] ), * ],
  GroupWithAction [ Group( [ f1, f2 ] ), * ],
  GroupWithAction [ Group( [ f1, f2 ] ), * ],
  GroupWithAction [ Group( [ f1, f2 ] ), * ],
  GroupWithAction [ Group( [ f1, f2 ] ), * ],
  GroupWithAction [ Group( [ f1, f2 ] ), * ],
  GroupWithAction [ Group( [ f1, f2 ] ), * ],
  GroupWithAction [ Group( [ f1, f2 ] ), * ],
  GroupWithAction [ Group( [ f1, f2 ] ), * ],
  GroupWithAction [ Group( [ f1, f2 ] ), * ] ]
\textcolor{blue}{gap> Length(aGA);}
10
\end{Verbatim}

We add the function \texttt{IsGwA()} implemented for testing is the given
structure is a group with action or not.

\begin{Verbatim}[frame=single, fontsize=\small, commandchars=\\\{\}]
\textcolor{blue}{gap> IsGwA(G);}
false
\textcolor{blue}{gap> GA := aGA[2];;}
\textcolor{blue}{gap> IsGwA(GA);}
true
\textcolor{blue}{gap> List(aGA, i -> IsGwA(i));}
[ true, true, true, true, true, true, true, true, true, true ]
\end{Verbatim}

We add the function \texttt{IsIdeal(GA,HA)}. By this we investigate if the
subobject $HA$ is an ideal of $GA$ or not. In addition, the functions
\texttt{NrIdealOnGwA(GA)}, \texttt{AllIdealOnGwA(GA)} are used to get all
ideals and number of all ideals of a given group with action $GA$,
respectively.

\begin{Verbatim}[frame=single, fontsize=\small, commandchars=\\\{\}]
\textcolor{blue}{gap> AllIdealOnGwA(GA);}
[ GroupWithAction [ Group( [ f1, f2 ] ), * ],
  GroupWithAction [ Group( [ f2 ] ), * ],
  GroupWithAction [ Group( <identity> of ... ), * ] ]
\textcolor{blue}{gap> HA := last[2];}
GroupWithAction [ Group( [ f2 ] ), * ]
\textcolor{blue}{gap> IsIdeal(GA,HA);}
true
\textcolor{blue}{gap> List(aGA,i -> NrIdealOnGwA(i));}
[ 3, 3, 3, 3, 3, 3, 3, 3, 3, 3 ]
\end{Verbatim}

We add the function \texttt{NilpotencyClassOfGwA(GA)} to obtain the
nilpotency class of a given group with action $GA$. The step-by-step
construction of the function is as follows:

\textbf{Step 1 : }We add the \texttt{Commutator()} function to obtain the
commutator subobject of a given group with action $GA.$

\textbf{Step 2 :} We obtain the lower central series of $GA$ by the function
\texttt{LowerCentralSeriesOfGwA(GA)}.

\begin{Verbatim}[frame=single, fontsize=\small, commandchars=\\\{\}]
\textcolor{blue}{gap> Commutator(GA,HA);}
GroupWithAction [ Group( [ <identity> of ..., f2, f2^2 ] ), * ]
\textcolor{blue}{gap> IsIdeal(GA,last);}
true
\textcolor{blue}{gap> LowerCentralSeriesOfGwA(GA);}
[ GroupWithAction [ Group( [ f1, f2 ] ), * ],
  GroupWithAction [ Group( [ <identity> of ..., f2, f2^2 ] ), * ] ]
\textcolor{blue}{gap> IsNilpotent(GA);}
false
\textcolor{blue}{gap> NilpotencyClassOfGwA(GA);}
0
\textcolor{blue}{gap> List(aGA,i -> IsNilpotent(i));}
[ false, false, false, false, false, false, false, false, false, false ]
\end{Verbatim}

Then we add the function \texttt{CenterOfGwA(GA)} to obtain the center of a
group with action $GA$ and add the function \texttt{IsSingularGwA(GA)} in
order to see if the group with action $GA$ is singular or not.

\begin{Verbatim}[frame=single, fontsize=\small, commandchars=\\\{\}]
\textcolor{blue}{gap> CA := CenterOfGwA(GA);}
GroupWithAction [ Group( [ <identity> of ... ] ), * ]
\textcolor{blue}{gap> IsIdeal(GA,CA);}
true
\textcolor{blue}{gap> IsAbelianGwA(GA);}
false
\textcolor{blue}{gap> IsAbelianGwA(CA);}
true
\textcolor{blue}{gap> List(aGA, i -> IsIdeal(i,CenterOfGwA(i)));}
[ true, true, true, true, true, true, true, true, true, true ]
\end{Verbatim}

We use the following steps to obtain the isomorphism classes of a group with
action;

\textbf{Step 1 :} We add the functions \texttt{GwAMorphism(phi)}, \texttt{%
IsGwAMorphism(phi)} to obtain and control group with action homomorphisms.

\textbf{Step 2 :}\ We add the function \texttt{IsIsomorphicGwA(GA,HA)} to
check if two groups with action are isomorphic or not.

\textbf{Step 3 :} \texttt{IsomorphicGwAFamily(aGA,GA)} gives the isomorphism
family of a given group with action.

\begin{Verbatim}[frame=single, fontsize=\small, commandchars=\\\{\}]
\textcolor{blue}{gap> IsIsomorphicGwA(GA,HA);}
false
\textcolor{blue}{gap> IsomorphicGwAFamily(aGA[1],aGA);}
1 => [ 1 ]
\textcolor{blue}{gap> IsomorphicGwAFamily(aGA[2],aGA);}
3 => [ 2, 3, 4 ]
\textcolor{blue}{gap> IsomorphicGwAFamily(aGA[5],aGA);}
2 => [ 5, 10 ]
\textcolor{blue}{gap> IsomorphicGwAFamily(aGA[6],aGA);}
3 => [ 6, 8, 9 ]
\textcolor{blue}{gap> IsomorphicGwAFamily(aGA[7],aGA);}
1 => [ 7 ]
\end{Verbatim}

We add the function \texttt{IsGwAC1(GA)} which controls if the group with
action $GA$ satisfies Condition 1 or not.

\begin{Verbatim}[frame=single, fontsize=\small, commandchars=\\\{\}]
\textcolor{blue}{gap> IsGwAC1(aGA[4]);}
true
\textcolor{blue}{gap> List(aGA, i -> IsGwAC1(i));}
[ true, true, true, true, false, false, true, false, false, false ]
\end{Verbatim}

We add the function \texttt{ActionTableOfGwA(GA)}, for constructing an
action table of the group with action $GA$,

\begin{Verbatim}[frame=single, fontsize=\small, commandchars=\\\{\}]
\textcolor{blue}{gap> MultiplicationTableOfGwA(aGA[4]);}
[ [ <identity> of ..., f1, f2, f1*f2, f2^2, f1*f2^2 ],
  [ <identity> of ..., f1, f2^2, f1*f2^2, f2, f1*f2 ],
  [ <identity> of ..., f1, f2, f1*f2, f2^2, f1*f2^2 ],
  [ <identity> of ..., f1, f2^2, f1*f2^2, f2, f1*f2 ],
  [ <identity> of ..., f1, f2, f1*f2, f2^2, f1*f2^2 ],
  [ <identity> of ..., f1, f2^2, f1*f2^2, f2, f1*f2 ] ]
\end{Verbatim}

In this session, we select the fourth group with action obtained from $%
S_{3}=\left\langle a,b\right\rangle =\left\{ e,a,b,ab,b^{2},ab^{2}\right\} $
whose action table is as follows:

\begin{equation*}
\begin{tabular}{c||cccccc}
$\varepsilon $ & $e$ & $a$ & $b$ & $ab$ & $b^{2}$ & $ab^{2}$ \\ \hline\hline
$e$ & $e$ & $a$ & $b$ & $ab$ & $b^{2}$ & $ab^{2}$ \\
$a$ & $e$ & $a$ & $b^{2}$ & $ab^{2}$ & $b$ & $ab$ \\
$b$ & $e$ & $a$ & $b$ & $ab$ & $b^{2}$ & $ab^{2}$ \\
$ab$ & $e$ & $a$ & $b^{2}$ & $ab^{2}$ & $b$ & $ab$ \\
$b^{2}$ & $e$ & $a$ & $b$ & $ab$ & $b^{2}$ & $ab^{2}$ \\
$ab^{2}$ & $e$ & $a$ & $b^{2}$ & $ab^{2}$ & $b$ & $ab$%
\end{tabular}%
\end{equation*}

Using above implementations, we have following table which gives some
algebraic properties of groups with action obtained from groups with order $%
<32$.

In the fifth row of the table the groups with action obtained from the Klein
four group, $C_{2}\times C_{2}$ are investigated. There are 10 groups with
action and 3 isomorphism families. Two of them satisfy Condition 1 and one
of them does not$.$One of them has $5$ ideals and two of them have $3$
ideals. There is one group with action with nilpotency class $2,$ one with
nilpotency class $1$ and one is not nilpotent.


\setlength{\tabcolsep}{2pt}
\begin{longtable}{|c|c|c|c|c|c|c|}
\hline\hline
GAP id & Name & |Gr$^{\bullet }$|  &
|Gr$^{\bullet }$/$\sim $| & |GrC1$^{\bullet }$/$\sim $| & |Ideals(Gr$^{\bullet }$/$\sim )$| & |Ni(Gr$^{\bullet }$/$\sim )$| \\ \hline\hline
\endhead
{\small [\ 1, 1 ]} & {\small I} & {\small 1} & {\small 1} & {\small 1} &
\multicolumn{1}{|l|}{%
\begin{tabular}{c}
{\small [\ 1, 1 ]}%
\end{tabular}%
} & \multicolumn{1}{|l|}{%
\begin{tabular}{c}
{\small [\ 0, 1 ]}%
\end{tabular}%
} \\ \hline
{\small [\ 2, 1 ]} & {\small C2} & {\small 1} & {\small 1} & {\small 1} &
\multicolumn{1}{|l|}{%
\begin{tabular}{c}
{\small [\ 2, 1 ]}%
\end{tabular}%
} & \multicolumn{1}{|l|}{%
\begin{tabular}{c}
{\small [\ 1, 1 ]}%
\end{tabular}%
} \\ \hline
{\small [\ 3, 1 ]} & {\small C3} & {\small 1} & {\small 1} & {\small 1} &
\multicolumn{1}{|l|}{%
\begin{tabular}{c}
{\small [\ 2, 1 ]}%
\end{tabular}%
} & \multicolumn{1}{|l|}{%
\begin{tabular}{c}
{\small [\ 1, 1 ]}%
\end{tabular}%
} \\ \hline
{\small [\ 4, 1 ]} & {\small C4} & {\small 2} & {\small 2} & {\small 2} &
\multicolumn{1}{|l|}{%
\begin{tabular}{c}
{\small [\ 3, 2 ]}%
\end{tabular}%
} & \multicolumn{1}{|l|}{%
\begin{tabular}{c}
{\small [\ 1, 1 ], [\ 2, 1 ]}%
\end{tabular}%
} \\ \hline
{\small [\ 4, 2 ]} & {\small C2xC2} & {\small 10} & {\small 3} & {\small 2} &
\multicolumn{1}{|l|}{%
\begin{tabular}{c}
{\small [\ 3, 2 ], [\ 5, 1 ]}%
\end{tabular}%
} & \multicolumn{1}{|l|}{%
\begin{tabular}{c}
{\small [\ 0, 1 ], [\ 1, 1 ], [\ 2, 1 ]}%
\end{tabular}%
} \\ \hline
{\small [\ 5, 1 ]} & {\small C5} & {\small 1} & {\small 1} & {\small 1} &
\multicolumn{1}{|l|}{%
\begin{tabular}{c}
{\small [\ 2, 1 ]}%
\end{tabular}%
} & \multicolumn{1}{|l|}{%
\begin{tabular}{c}
{\small [\ 1, 1 ]}%
\end{tabular}%
} \\ \hline
{\small [\ 6, 1 ]} & {\small S3} & {\small 10} & {\small 5} & {\small 3} &
\multicolumn{1}{|l|}{%
\begin{tabular}{c}
{\small [ 3, 5 ]}%
\end{tabular}%
} & \multicolumn{1}{|l|}{%
\begin{tabular}{c}
{\small [ 0, 5 ]}%
\end{tabular}%
} \\ \hline
{\small [\ 6, 2 ]} & {\small C6} & {\small 2} & {\small 2} & {\small 2} &
\multicolumn{1}{|l|}{%
\begin{tabular}{c}
{\small [ 3, 1 ], [ 4, 1 ]}%
\end{tabular}%
} & \multicolumn{1}{|l|}{%
\begin{tabular}{c}
{\small [ 0, 1 ], [ 1, 1 ]}%
\end{tabular}%
} \\ \hline
{\small [\ 7, 1 ]} & {\small C7} & {\small 1} & {\small 1} & {\small 1} &
\multicolumn{1}{|l|}{%
\begin{tabular}{c}
{\small [\ 2, 1 ]}%
\end{tabular}%
} & \multicolumn{1}{|l|}{%
\begin{tabular}{c}
{\small [\ 1, 1 ]}%
\end{tabular}%
} \\ \hline
{\small [\ 8, 1 ]} & {\small C8} & {\small 4} & {\small 4} & {\small 4} &
\multicolumn{1}{|l|}{%
\begin{tabular}{c}
{\small [ 4, 4 ]}%
\end{tabular}%
} & \multicolumn{1}{|l|}{%
\begin{tabular}{c}
{\small [ 0, 2 ], [ 1, 1 ], [ 2, 1 ]}%
\end{tabular}%
} \\ \hline
{\small [\ 8, 2 ]} & {\small C4xC2} & {\small 32} & {\small 15} & {\small 11} &
\multicolumn{1}{|l|}{%
\begin{tabular}{l}
{\small [ 4, 2 ], [ 5, 2 ], [ 6, 9 ], } \\
{\small [ 8, 2 ]}%
\end{tabular}%
} & \multicolumn{1}{|l|}{%
\begin{tabular}{c}
{\small [ 0, 4 ], [ 1, 1 ], [ 2, 10 ]}%
\end{tabular}%
} \\ \hline
{\small [\ 8, 3 ]} & {\small D8} & {\small 36} & {\small 16} & {\small 11} &
\multicolumn{1}{|l|}{%
\begin{tabular}{c}
{\small [ 4, 6 ], [ 6, 10 ]}%
\end{tabular}%
} & \multicolumn{1}{|l|}{%
\begin{tabular}{c}
{\small [ 0, 6 ], [ 2, 10 ]}%
\end{tabular}%
} \\ \hline
{\small [\ 8, 4 ]} & {\small Q8} & {\small 52} & {\small 10} & {\small 7} &
\multicolumn{1}{|l|}{%
\begin{tabular}{c}
{\small [ 4, 4 ], [ 6, 6 ]}%
\end{tabular}%
} & \multicolumn{1}{|l|}{%
\begin{tabular}{c}
{\small [ 0, 4 ], [ 2, 6 ]}%
\end{tabular}%
} \\ \hline
{\small [\ 8, 5 ]} & {\small C2xC2xC2} & {\small 736} & {\small 14} & {\small 6} &
\multicolumn{1}{|l|}{%
\begin{tabular}{c}
{\small [ 3, 1 ], [ 4, 2 ], [ 5, 1 ], } \\
\multicolumn{1}{l}{\small [ 6, 7 ], [ 7, 1 ], [ 8, 1 ],} \\
\multicolumn{1}{l}{\small [ 16, 1 ]}%
\end{tabular}%
} & \multicolumn{1}{|l|}{%
\begin{tabular}{c}
{\small [ 0, 8 ], [ 1, 1 ], [ 2, 5 ]}%
\end{tabular}%
} \\ \hline
{\small [\ 9, 1 ]} & {\small C9} & {\small 3} & {\small 2} & {\small 2} &
\multicolumn{1}{|l|}{%
\begin{tabular}{c}
{\small [ 3, 2 ]}%
\end{tabular}%
} & \multicolumn{1}{|l|}{%
\begin{tabular}{c}
{\small [ 1, 1 ], [ 2, 1 ]}%
\end{tabular}%
} \\ \hline
{\small [\ 9, 2 ]} & {\small C3xC3} & {\small 33} & {\small 3} & {\small 2} &
\multicolumn{1}{|l|}{%
\begin{tabular}{c}
{\small [ 3, 2 ], [ 6, 1 ]}%
\end{tabular}%
} & \multicolumn{1}{|l|}{%
\begin{tabular}{c}
{\small [ 0, 1 ], [ 1, 1 ], [ 2, 1 ]}%
\end{tabular}%
} \\ \hline
{\small [\ 10, 1 ]} & {\small D10} & {\small 26} & {\small 7} & {\small 3} &
\multicolumn{1}{|l|}{%
\begin{tabular}{c}
{\small [ 3, 7 ]}%
\end{tabular}%
} & \multicolumn{1}{|l|}{%
\begin{tabular}{c}
{\small [ 0, 7 ]}%
\end{tabular}%
} \\ \hline
{\small [\ 10, 2 ]} & {\small C10} & {\small 2} & {\small 2} & {\small 2} &
\multicolumn{1}{|l|}{%
\begin{tabular}{c}
{\small [ 3, 1 ], [ 4, 1 ]}%
\end{tabular}%
} & \multicolumn{1}{|l|}{%
\begin{tabular}{c}
{\small [ 0, 1 ], [ 1, 1 ]}%
\end{tabular}%
} \\ \hline
{\small [\ 11, 1 ]} & {\small C11} & {\small 1} & {\small 1} & {\small 1} &
\multicolumn{1}{|l|}{%
\begin{tabular}{c}
{\small [\ 2, 1 ]}%
\end{tabular}%
} & \multicolumn{1}{|l|}{%
\begin{tabular}{c}
{\small [\ 1, 1 ]}%
\end{tabular}%
} \\ \hline
{\small [\ 12, 1 ]} & {\small C3:C4} & {\small 20} & {\small 10} & {\small 6} &
\multicolumn{1}{|l|}{%
\begin{tabular}{c}
{\small [ 5, 10 ]}%
\end{tabular}%
} & \multicolumn{1}{|l|}{%
\begin{tabular}{c}
{\small [ 0, 10 ]}%
\end{tabular}%
} \\ \hline
{\small [\ 12, 2 ]} & {\small C12} & {\small 4} & {\small 4} & {\small 4} &
\multicolumn{1}{|l|}{%
\begin{tabular}{c}
{\small [ 5, 2 ], [ 6, 2 ]}%
\end{tabular}%
} & \multicolumn{1}{|l|}{%
\begin{tabular}{c}
{\small [ 0, 2 ], [ 1, 1 ], [ 2, 1 ]}%
\end{tabular}%
} \\ \hline
{\small [\ 12, 3 ]} & {\small A4} & {\small 33} & {\small 8} & {\small 4} &
\multicolumn{1}{|l|}{%
\begin{tabular}{c}
{\small [ 3, 8 ]}%
\end{tabular}%
} & \multicolumn{1}{|l|}{%
\begin{tabular}{c}
{\small [ 0, 8 ]}%
\end{tabular}%
} \\ \hline
{\small [\ 12, 4 ]} & {\small D12} & {\small 64} & {\small 19} & {\small 7} &
\multicolumn{1}{|l|}{%
\begin{tabular}{l}
{\small [ 4, 3 ], [ 5, 10 ], [ 6, 1 ], } \\
{\small [ 7, 5 ]}%
\end{tabular}%
} & \multicolumn{1}{|l|}{%
\begin{tabular}{c}
{\small [ 0, 19 ]}%
\end{tabular}%
} \\ \hline
{\small [\ 12, 5 ]} & {\small C6xC2} & {\small 48} & {\small 11} & {\small 5} &
\multicolumn{1}{|l|}{%
\begin{tabular}{c}
{\small [ 2, 1 ], [ 3, 1 ], [ 4, 3 ],} \\
\multicolumn{1}{l}{\small [ 5, 2 ], [ 6, 2 ],} \\
\multicolumn{1}{l}{\small [ 7, 1 ], [ 10, 1 ]}%
\end{tabular}%
} & \multicolumn{1}{|l|}{%
\begin{tabular}{c}
{\small [ 0, 9 ], [ 1, 1 ], [ 2, 1 ]}%
\end{tabular}%
} \\ \hline
{\small [\ 13, 1 ]} & {\small C13} & {\small 1} & {\small 1} & {\small 1} &
\multicolumn{1}{|l|}{%
\begin{tabular}{c}
{\small [\ 2, 1 ]}%
\end{tabular}%
} & \multicolumn{1}{|l|}{%
\begin{tabular}{c}
{\small [\ 1, 1 ]}%
\end{tabular}%
} \\ \hline
{\small [\ 14, 1 ]} & {\small D14} & {\small 50} & {\small 9} & {\small 3} &
\multicolumn{1}{|l|}{%
\begin{tabular}{c}
{\small [ 3, 9 ]}%
\end{tabular}%
} & \multicolumn{1}{|l|}{%
\begin{tabular}{c}
{\small [ 0, 9 ]}%
\end{tabular}%
} \\ \hline
{\small [\ 14, 2 ]} & {\small C14} & {\small 2} & {\small 2} & {\small 2} &
\multicolumn{1}{|l|}{%
\begin{tabular}{c}
{\small [ 3, 1 ], [ 4, 1 ]}%
\end{tabular}%
} & \multicolumn{1}{|l|}{%
\begin{tabular}{c}
{\small [ 0, 1 ], [ 1, 1 ]}%
\end{tabular}%
} \\ \hline
{\small [\ 15, 1 ]} & {\small C15} & {\small 1} & {\small 1} & {\small 1} &
\multicolumn{1}{|l|}{%
\begin{tabular}{c}
{\small [ 4, 1 ]}%
\end{tabular}%
} & \multicolumn{1}{|l|}{%
\begin{tabular}{c}
{\small [ 1, 1 ]}%
\end{tabular}%
} \\ \hline
{\small [\ 16, 1 ]} & {\small C16} & {\small 8} & {\small 6} & {\small 5} &
\multicolumn{1}{|l|}{%
\begin{tabular}{c}
{\small [ 5, 6 ]}%
\end{tabular}%
} & \multicolumn{1}{|l|}{%
\begin{tabular}{c}
{\small [ 0, 3 ], [ 1, 1 ], [ 2, 2 ]}%
\end{tabular}%
} \\ \hline
{\small [\ 16, 2 ]} & {\small C4xC4} & {\small 832} & {\small 73} & {\small 54} &
\multicolumn{1}{|l|}{%
\begin{tabular}{l}
{\small [ 4, 4 ], [ 5, 7 ], [ 6, 6 ], } \\
{\small [ 7, 4 ], [ 9, 33 ], [ 11, 16 ], } \\
{\small [ 13, 2 ], [ 15, 1 ]}%
\end{tabular}%
} & \multicolumn{1}{|l|}{%
\begin{tabular}{c}
{\small [ 0, 20 ], [ 1, 1 ], [ 2, 52 ]}%
\end{tabular}%
} \\ \hline
{\small [\ 16, 3 ]} & {\small (C4xC2):C2} & {\small 640} & {\small 168} & {\small 138} &
\multicolumn{1}{|l|}{%
\begin{tabular}{l}
{\small [ 4, 12 ], [ 5, 9 ], [ 6, 6 ], } \\
{\small [ 7, 5 ], [ 9, 116 ], [ 11, 20 ]}%
\end{tabular}%
} & \multicolumn{1}{|l|}{%
\begin{tabular}{c}
{\small [ 0, 32 ], [ 2, 136 ]}%
\end{tabular}%
} \\ \hline
{\small [\ 16, 4 ]} & {\small C4:C4} & {\small 448} & {\small 161} & {\small 138} &
\multicolumn{1}{|l|}{%
\begin{tabular}{l}
{\small [ 5, 3 ], [ 6, 8 ], [ 7, 8 ], } \\
{\small [ 8, 4 ], [ 9, 118 ], [ 11, 20 ]}%
\end{tabular}%
} & \multicolumn{1}{|l|}{%
\begin{tabular}{c}
{\small [ 0, 25 ], [ 2, 136 ]}%
\end{tabular}%
} \\ \hline
{\small [\ 16, 5 ]} & {\small C8xC2} & {\small 128} & {\small 56} & {\small 40} &
\multicolumn{1}{|l|}{%
\begin{tabular}{l}
{\small [ 5, 4 ], [ 6, 4 ], [ 7, 26 ], } \\
{\small [ 9, 18 ], [ 11, 4 ]}%
\end{tabular}%
} & \multicolumn{1}{|l|}{%
\begin{tabular}{c}
{\small [ 0, 44 ], [ 1, 1 ], [ 2, 11 ]}%
\end{tabular}%
} \\ \hline
{\small [\ 16, 6 ]} & {\small C8:C2} & {\small 128} & {\small 56} & {\small 40} &
\multicolumn{1}{|l|}{%
\begin{tabular}{l}
{\small [ 5, 4 ], [ 6, 8 ], [ 7, 24 ],} \\
{\small [ 9, 20 ]}%
\end{tabular}%
} & \multicolumn{1}{|l|}{%
\begin{tabular}{c}
{\small [ 0, 46 ], [ 2, 10 ]}%
\end{tabular}%
} \\ \hline
{\small [\ 16, 7 ]} & {\small D16} & {\small 256} & {\small 63} & {\small 45} &
\multicolumn{1}{|l|}{%
\begin{tabular}{c}
{\small [ 5, 13 ], [ 7, 50 ]}%
\end{tabular}%
} & \multicolumn{1}{|l|}{%
\begin{tabular}{c}
{\small [ 0, 63 ]}%
\end{tabular}%
} \\ \hline
{\small [\ 16, 8 ]} & {\small QD16} & {\small 144} & {\small 88} & {\small 80} &
\multicolumn{1}{|l|}{%
\begin{tabular}{c}
{\small [ 7, 88 ]}%
\end{tabular}%
} & \multicolumn{1}{|l|}{%
\begin{tabular}{c}
{\small [ 0, 88 ]}%
\end{tabular}%
} \\ \hline
{\small [\ 16, 9 ]} & {\small Q16} & {\small 192} & {\small 57} & {\small 45} &
\multicolumn{1}{|l|}{%
\begin{tabular}{c}
{\small [ 5, 7 ], [ 7, 50 ]}%
\end{tabular}%
} & \multicolumn{1}{|l|}{%
\begin{tabular}{c}
{\small [ 0, 57 ]}%
\end{tabular}%
} \\ \hline
{\small [\ 16, 10 ]} & {\small C4xC2xC2} & {\small 14912} & {\small 404} & {\small 105} &
\multicolumn{1}{|l|}{%
\begin{tabular}{l}
{\small [ 4, 7 ], [ 5, 10 ], [ 6, 6 ],} \\
{\small [ 7, 129 ], [ 8, 92 ], [ 9, 75 ],} \\
{\small [ 10, 11 ], [ 11, 31 ], [ 13, 4 ],} \\
{\small [ 17, 28 ], [ 19, 9 ], [ 27, 2 ]}%
\end{tabular}%
} & \multicolumn{1}{|l|}{%
\begin{tabular}{c}
{\small [ 0, 300 ], [ 1, 1 ], [ 2, 103 ]}%
\end{tabular}%
} \\ \hline
{\small [\ 16, 11 ]} & {\small C2xD8} & {\small 7744} & {\small 578} & {\small 166} &
\multicolumn{1}{|l|}{%
\begin{tabular}{l}
{\small [ 4, 14 ], [ 5, 25 ], [ 6, 16 ],} \\
{\small [ 7, 157 ], [ 8, 149 ], [ 9, 113 ],} \\
{\small [ 10, 2 ], [ 11, 16 ], [ 17, 76 ],} \\
{\small [ 19, 10 ]}%
\end{tabular}%
} & \multicolumn{1}{|l|}{%
\begin{tabular}{c}
{\small [ 0, 426 ], [ 2, 152 ]}%
\end{tabular}%
} \\ \hline
{\small [\ 16, 12 ]} & {\small C2 x Q8} & {\small 9536} & {\small 275} & {\small 80} &
\multicolumn{1}{|l|}{%
\begin{tabular}{l}
{\small [ 4, 4 ], [ 5, 14 ], [ 6, 14 ], } \\
{\small [ 7, 67 ], [ 8, 77 ], [ 9, 51 ], } \\
{\small [ 10, 2 ], [ 11, 8 ], [ 17, 32 ], } \\
{\small [ 19, 6 ]}%
\end{tabular}%
} & \multicolumn{1}{|l|}{%
\begin{tabular}{c}
{\small [ 0, 209 ], [ 2, 66 ]}%
\end{tabular}%
} \\ \hline
{\small [\ 16, 13 ]} & {\small (C4xC2):C2} & {\small 1856} & {\small 232} & {\small 128} &
\multicolumn{1}{|l|}{%
\begin{tabular}{l}
{\small [ 7, 40 ], [ 8, 64 ], [ 9, 24 ], } \\
{\small [ 17, 104 ]}%
\end{tabular}%
} & \multicolumn{1}{|l|}{%
\begin{tabular}{c}
{\small [ 0, 128 ], [ 2, 104 ]}%
\end{tabular}%
} \\ \hline
{\small [\ 17, 1 ]} & {\small C17} & {\small 1} & {\small 1} & {\small 1} &
\multicolumn{1}{|l|}{%
\begin{tabular}{c}
{\small [\ 2, 1 ]}%
\end{tabular}%
} & \multicolumn{1}{|l|}{%
\begin{tabular}{c}
{\small [\ 1, 1 ]}%
\end{tabular}%
} \\ \hline
{\small [\ 18, 1 ]} & {\small D18} & {\small 82} & {\small 14} & {\small 3} &
\multicolumn{1}{|l|}{%
\begin{tabular}{c}
{\small [ 4, 14 ]}%
\end{tabular}%
} & \multicolumn{1}{|l|}{%
\begin{tabular}{c}
{\small [ 0, 14 ]}%
\end{tabular}%
} \\ \hline
{\small [\ 18, 2 ]} & {\small C16} & {\small 6} & {\small 4} & {\small 3} &
\multicolumn{1}{|l|}{%
\begin{tabular}{c}
{\small [ 4, 2 ], [ 6, 2 ]}%
\end{tabular}%
} & \multicolumn{1}{|l|}{%
\begin{tabular}{c}
{\small [ 0, 2 ], [ 1, 1 ], [ 2, 1 ]}%
\end{tabular}%
} \\ \hline
{\small [\ 18, 3 ]} & {\small C3xS3} & {\small 24} & {\small 12} & {\small 7} &
\multicolumn{1}{|l|}{%
\begin{tabular}{c}
{\small [ 4, 1 ], [ 5, 6 ], [ 6, 5 ]}%
\end{tabular}%
} & \multicolumn{1}{|l|}{%
\begin{tabular}{c}
{\small [ 0, 12 ]}%
\end{tabular}%
} \\ \hline
{\small [\ 18, 4 ]} & {\small (C3xC3):C2} & {\small 4510} & {\small 41} & {\small 6} &
\multicolumn{1}{|l|}{%
\begin{tabular}{l}
{\small [ 3, 4 ], [ 4, 23 ], [ 5, 9 ], } \\
{\small [ 7, 5 ]}%
\end{tabular}%
} & \multicolumn{1}{|l|}{%
\begin{tabular}{c}
{\small [ 0, 41 ]}%
\end{tabular}%
} \\ \hline
{\small [\ 18, 5 ]} & {\small C6xC3} & {\small 78} & {\small 7} & {\small 4} &
\multicolumn{1}{|l|}{%
\begin{tabular}{l}
{\small [ 4, 2 ], [ 6, 3 ], [ 7, 1 ], } \\
{\small [ 12, 1 ]}%
\end{tabular}%
} & \multicolumn{1}{|l|}{%
\begin{tabular}{c}
{\small [ 0, 5 ], [ 1, 1 ], [ 2, 1 ]}%
\end{tabular}%
} \\ \hline
{\small [\ 19, 1 ]} & {\small C19} & {\small 1} & {\small 1} & {\small 1} &
\multicolumn{1}{|l|}{%
\begin{tabular}{c}
{\small [\ 2, 1 ]}%
\end{tabular}%
} & \multicolumn{1}{|l|}{%
\begin{tabular}{c}
{\small [\ 1, 1 ]}%
\end{tabular}%
} \\ \hline
{\small [\ 20, 1 ]} & {\small \ Q20} & {\small 72} & {\small 16} & {\small 7} &
\multicolumn{1}{|l|}{%
\begin{tabular}{c}
{\small [ 4, 2 ], [ 5, 14 ]}%
\end{tabular}%
} & \multicolumn{1}{|l|}{%
\begin{tabular}{c}
{\small [ 0, 16 ]}%
\end{tabular}%
} \\ \hline
{\small [\ 20, 2 ]} & {\small \ C20} & {\small 8} & {\small 6} & {\small 5} &
\multicolumn{1}{|l|}{%
\begin{tabular}{c}
{\small [ 4, 2 ], [ 5, 2 ], [ 6, 2 ]}%
\end{tabular}%
} & \multicolumn{1}{|l|}{%
\begin{tabular}{c}
{\small [ 0, 4 ], [ 1, 1 ], [ 2, 1 ]}%
\end{tabular}%
} \\ \hline
{\small [\ 20, 3 ]} & {\small \ C5:C4} & {\small 36} & {\small 9} & {\small 5} &
\multicolumn{1}{|l|}{%
\begin{tabular}{c}
{\small [ 4, 9 ]}%
\end{tabular}%
} & \multicolumn{1}{|l|}{%
\begin{tabular}{c}
{\small [ 0, 9 ]}%
\end{tabular}%
} \\ \hline
{\small [\ 20, 4 ]} & {\small D20} & {\small 144} & {\small 25} & {\small 7} &
\multicolumn{1}{|l|}{%
\begin{tabular}{l}
{\small [ 4, 3 ], [ 5, 14 ], [ 6, 1 ],} \\
{\small [ 7, 7 ]}%
\end{tabular}%
} & \multicolumn{1}{|l|}{%
\begin{tabular}{c}
{\small [ 0, 25 ]}%
\end{tabular}%
} \\ \hline
{\small [\ 20, 5 ]} & {\small C10xC2} & {\small 40} & {\small 9} & {\small 4} &
\multicolumn{1}{|l|}{%
\begin{tabular}{l}
{\small [ 4, 3 ], [ 5, 2 ], [ 6, 2 ], } \\
{\small [ 7, 1 ], [ 10, 1 ]}%
\end{tabular}%
} & \multicolumn{1}{|l|}{%
\begin{tabular}{c}
{\small [ 0, 7 ], [ 1, 1 ], [ 2, 1 ]}%
\end{tabular}%
} \\ \hline
{\small [\ 21, 1 ]} & {\small C7:C3} & {\small 57} & {\small 10} & {\small 4} &
\multicolumn{1}{|l|}{%
\begin{tabular}{c}
{\small [ 3, 10 ]}%
\end{tabular}%
} & \multicolumn{1}{|l|}{%
\begin{tabular}{c}
{\small [ 0, 10 ]}%
\end{tabular}%
} \\ \hline
{\small [\ 21, 2 ]} & {\small C21} & {\small 3} & {\small 2} & {\small 2} &
\multicolumn{1}{|l|}{%
\begin{tabular}{c}
{\small [ 3, 1 ], [ 4, 1 ]}%
\end{tabular}%
} & \multicolumn{1}{|l|}{%
\begin{tabular}{c}
{\small [ 0, 1 ], [ 1, 1 ]}%
\end{tabular}%
} \\ \hline
{\small [ 22, 1 ]} & {\small D22} & {\small 122} & {\small 13} & {\small 3} &
\multicolumn{1}{|l|}{%
\begin{tabular}{c}
{\small [ 3, 13 ]}%
\end{tabular}%
} & \multicolumn{1}{|l|}{%
\begin{tabular}{c}
{\small [ 0, 13 ]}%
\end{tabular}%
} \\ \hline
{\small [\ 22, 2 ]} & {\small C22} & {\small 2} & {\small 2} & {\small 2} &
\multicolumn{1}{|l|}{%
\begin{tabular}{c}
{\small [ 3, 1 ], [ 4, 1 ]}%
\end{tabular}%
} & \multicolumn{1}{|l|}{%
\begin{tabular}{c}
{\small [ 0, 1 ], [ 1, 1 ]}%
\end{tabular}%
} \\ \hline
{\small [\ 23, 1 ]} & {\small C23} & {\small 1} & {\small 1} & {\small 1} &
\multicolumn{1}{|l|}{%
\begin{tabular}{c}
{\small [\ 2, 1 ]}%
\end{tabular}%
} & \multicolumn{1}{|l|}{%
\begin{tabular}{c}
{\small [\ 1, 1 ]}%
\end{tabular}%
} \\ \hline
{\small [\ 24, 1 ]} & {\small C3:C8} & {\small 40} & {\small 20} & {\small 12} &
\multicolumn{1}{|l|}{%
\begin{tabular}{c}
{\small [ 7, 20 ]}%
\end{tabular}%
} & \multicolumn{1}{|l|}{%
\begin{tabular}{c}
{\small [ 0, 20 ]}%
\end{tabular}%
} \\ \hline
{\small [\ 24, 2 ]} & {\small C24} & {\small 8} & {\small 8} & {\small 8} &
\multicolumn{1}{|l|}{%
\begin{tabular}{c}
{\small [ 7, 4 ], [ 8, 4 ]}%
\end{tabular}%
} & \multicolumn{1}{|l|}{%
\begin{tabular}{c}
{\small [ 0, 6 ], [ 1, 1 ], [ 2, 1 ]}%
\end{tabular}%
} \\ \hline
{\small [\ 24, 3 ]} & {\small SL(2,3)} & {\small 33} & {\small 8} & {\small 2} &
\multicolumn{1}{|l|}{%
\begin{tabular}{c}
{\small [ 4, 8 ]}%
\end{tabular}%
} & \multicolumn{1}{|l|}{%
\begin{tabular}{c}
{\small [ 0, 8 ]}%
\end{tabular}%
} \\ \hline
{\small [\ 24, 4 ]} & {\small C3:Q8} & {\small 448} & {\small 92} & {\small 49} &
\multicolumn{1}{|l|}{%
\begin{tabular}{l}
{\small [ 6, 6 ], [ 7, 20 ], [ 8, 16 ], } \\
{\small [ 9, 50 ]}%
\end{tabular}%
} & \multicolumn{1}{|l|}{%
\begin{tabular}{c}
{\small [ 0, 92 ]}%
\end{tabular}%
} \\ \hline
{\small [\ 24, 5 ]} & {\small C4xS3} & {\small 256} & {\small 112} & {\small 80} &
\multicolumn{1}{|l|}{%
\begin{tabular}{l}
{\small [ 8, 28 ], [ 9, 70 ], [ 10, 4 ], } \\
{\small [ 11, 10 ]}%
\end{tabular}%
} & \multicolumn{1}{|l|}{%
\begin{tabular}{c}
{\small [ 0, 112 ]}%
\end{tabular}%
} \\ \hline
{\small [\ 24, 6 ]} & {\small D24} & {\small 576} & {\small 106} & {\small 49} &
\multicolumn{1}{|l|}{%
\begin{tabular}{l}
{\small [ 6, 10 ], [ 7, 30 ], [ 8, 16 ], } \\
{\small [ 9, 50 ]}%
\end{tabular}%
} & \multicolumn{1}{|l|}{%
\begin{tabular}{c}
{\small [ 0, 106 ]}%
\end{tabular}%
} \\ \hline
{\small [\ 24, 7 ]} & {\small \ C2x(C3:C4)} & {\small 512} & {\small 98} & {\small 49} &
\multicolumn{1}{|l|}{%
\begin{tabular}{l}
{\small [ 6, 2 ], [ 7, 14 ], [ 8, 15 ], } \\
{\small [ 9, 50 ], [ 10, 2 ], [ 11, 5 ], } \\
{\small [ 13, 10 ]}%
\end{tabular}%
} & \multicolumn{1}{|l|}{%
\begin{tabular}{c}
{\small [ 0, 98 ]}%
\end{tabular}%
} \\ \hline
{\small [\ 24, 8 ]} & {\small (C6xC2):C2} & {\small 256} & {\small 112} & {\small 80} &
\multicolumn{1}{|l|}{%
\begin{tabular}{c}
{\small [ 8, 32 ], [ 9, 80 ]}%
\end{tabular}%
} & \multicolumn{1}{|l|}{%
\begin{tabular}{c}
{\small [ 0, 112 ]}%
\end{tabular}%
} \\ \hline
{\small [\ 24, 9 ]} & {\small C12xC2} & {\small 128} & {\small 53} & {\small 38} &
\multicolumn{1}{|l|}{%
\begin{tabular}{l}
{\small [ 6, 2 ], [ 7, 6 ], [ 8, 3 ], } \\
{\small [ 9, 24 ], [ 10, 2 ], [ 11, 3 ], } \\
{\small [ 12, 9 ], [ 13, 2 ], [ 16, 2 ]}%
\end{tabular}%
} & \multicolumn{1}{|l|}{%
\begin{tabular}{c}
{\small [ 0, 42 ], [ 1, 1 ], [ 2, 10 ]}%
\end{tabular}%
} \\ \hline
{\small [\ 24, 10 ]} & {\small \ C3xD8} & {\small 144} & {\small 58} & {\small 38} &
\multicolumn{1}{|l|}{%
\begin{tabular}{l}
{\small [ 6, 10 ], [ 7, 6 ], [ 8, 6 ], } \\
{\small [ 9, 26 ], [ 12, 10 ]}%
\end{tabular}%
} & \multicolumn{1}{|l|}{%
\begin{tabular}{c}
{\small [ 0, 48 ], [ 2, 10 ]}%
\end{tabular}%
} \\ \hline
{\small [\ 24, 11 ]} & {\small C3xQ8} & {\small 240} & {\small 32} & {\small 18} &
\multicolumn{1}{|l|}{%
\begin{tabular}{l}
{\small [ 3, 1 ], [ 4, 1 ], [ 6, 6 ], } \\
{\small [ 7, 4 ], [ 8, 4 ], [ 9, 10 ], } \\
{\small [ 12, 6 ]}%
\end{tabular}%
} & \multicolumn{1}{|l|}{%
\begin{tabular}{c}
{\small [ 0, 26 ], [ 2, 6 ]}%
\end{tabular}%
} \\ \hline
{\small [\ 24, 12 ]} & {\small S4} & {\small 58} & {\small 11} & {\small 4} &
\multicolumn{1}{|l|}{%
\begin{tabular}{c}
{\small [ 4, 11 ]}%
\end{tabular}%
} & \multicolumn{1}{|l|}{%
\begin{tabular}{c}
{\small [ 0, 11 ]}%
\end{tabular}%
} \\ \hline
{\small [\ 24, 13 ]} & {\small C2xA4} & {\small 42} & {\small 10} & {\small 6} &
\multicolumn{1}{|l|}{%
\begin{tabular}{c}
{\small [ 4, 1 ], [ 5, 1 ], [ 6, 8 ]}%
\end{tabular}%
} & \multicolumn{1}{|l|}{%
\begin{tabular}{c}
{\small [ 0, 10 ]}%
\end{tabular}%
} \\ \hline
{\small [\ 24, 14 ]} & {\small C2xC2xS3} & {\small 7168} & {\small 196} & {\small 46} &
\multicolumn{1}{|l|}{%
\begin{tabular}{l}
{\small [ 3, 4 ], [ 4, 7 ], [ 5, 8 ], } \\
{\small [ 6, 4 ], [ 7, 39 ], [ 8, 28 ], } \\
{\small [ 9, 71 ], [ 10, 9 ], [ 11, 15 ], } \\
{\small [ 13, 5 ], [ 18, 1 ], [ 21, 5 ]}%
\end{tabular}%
} & \multicolumn{1}{|l|}{%
\begin{tabular}{c}
{\small [ 0, 196 ]}%
\end{tabular}%
} \\ \hline
{\small [\ 24, 15 ]} & {\small C6xC2xC2} & {\small 6636} & {\small 69} & {\small 17} &
\multicolumn{1}{|l|}{%
\begin{tabular}{l}
{\small [ 3, 1 ], [ 4, 4 ], [ 5, 4 ], } \\
{\small [ 6, 6 ], [ 7, 16 ], [ 8, 5 ], } \\
{\small [ 9, 15 ], [ 10, 3 ], [ 11, 3 ], } \\
{\small [ 12, 7 ], [ 13, 1 ], [ 14, 1 ], } \\
{\small [ 16, 1 ], [ 21, 1 ], [ 32, 1 ]}%
\end{tabular}%
} & \multicolumn{1}{|l|}{%
\begin{tabular}{c}
{\small [ 0, 63 ], [ 1, 1 ], [ 2, 5 ]}%
\end{tabular}%
} \\ \hline
{\small [\ 25, 1 ]} & {\small C25} & {\small 5} & {\small 2} & {\small 2} &
\multicolumn{1}{|l|}{%
\begin{tabular}{c}
{\small [ 3, 2 ]}%
\end{tabular}%
} & \multicolumn{1}{|l|}{%
\begin{tabular}{c}
{\small [ 1, 1 ], [ 2, 1 ]}%
\end{tabular}%
} \\ \hline
{\small [\ 25, 2 ]} & {\small C5xC5} & {\small 145} & {\small 3} & {\small 2} &
\multicolumn{1}{|l|}{%
\begin{tabular}{c}
{\small [ 3, 2 ], [ 8, 1 ]}%
\end{tabular}%
} & \multicolumn{1}{|l|}{%
\begin{tabular}{c}
{\small [ 0, 1 ], [ 1, 1 ], [ 2, 1 ]}%
\end{tabular}%
} \\ \hline
{\small [\ 26, 1 ]} & {\small D26} & {\small 170} & {\small 15} & {\small 3} &
\multicolumn{1}{|l|}{%
\begin{tabular}{c}
{\small [ 3, 15 ]}%
\end{tabular}%
} & \multicolumn{1}{|l|}{%
\begin{tabular}{c}
{\small [ 0, 15 ]}%
\end{tabular}%
} \\ \hline
{\small [\ 26, 2 ]} & {\small C26} & {\small 2} & {\small 2} & {\small 2} &
\multicolumn{1}{|l|}{%
\begin{tabular}{c}
{\small [ 3, 1 ], [ 4, 1 ]}%
\end{tabular}%
} & \multicolumn{1}{|l|}{%
\begin{tabular}{c}
{\small [ 0, 1 ], [ 1, 1 ]}%
\end{tabular}%
} \\ \hline
{\small [\ 27, 1 ]} & {\small C27} & {\small 9} & {\small 3} & {\small 2} &
\multicolumn{1}{|l|}{%
\begin{tabular}{c}
{\small [ 4, 3 ]}%
\end{tabular}%
} & \multicolumn{1}{|l|}{%
\begin{tabular}{c}
{\small [ 0, 1 ], [ 1, 1 ], [ 2, 1 ]}%
\end{tabular}%
} \\ \hline
{\small [\ 27, 2 ]} & {\small \ C9xC3} & {\small 297} & {\small 31} & {\small 16} &
\multicolumn{1}{|l|}{%
\begin{tabular}{l}
{\small [ 4, 14 ], [ 5, 3 ], [ 7, 12 ], } \\
{\small [ 10, 2 ]}%
\end{tabular}%
} & \multicolumn{1}{|l|}{%
\begin{tabular}{c}
{\small [ 0, 17 ], [ 1, 1 ], [ 2, 13 ]}%
\end{tabular}%
} \\ \hline
{\small [\ 27, 3 ]} & {\small (C3xC3):C3} & {\small 2673} & {\small 35} & {\small 15} &
\multicolumn{1}{|l|}{%
\begin{tabular}{c}
{\small [ 4, 23 ], [ 7, 12 ]}%
\end{tabular}%
} & \multicolumn{1}{|l|}{%
\begin{tabular}{c}
{\small [ 0, 23 ], [ 2, 12 ]}%
\end{tabular}%
} \\ \hline
{\small [\ 27, 4 ]} & {\small C9:C3} & {\small 297} & {\small 48} & {\small 27} &
\multicolumn{1}{|l|}{%
\begin{tabular}{c}
{\small [ 4, 27 ], [ 7, 21 ]}%
\end{tabular}%
} & \multicolumn{1}{|l|}{%
\begin{tabular}{c}
{\small [ 0, 27 ], [ 2, 21 ]}%
\end{tabular}%
} \\ \hline
{\small [\ 28, 1 ]} & {\small C7:C4} & {\small 100} & {\small 18} & {\small 6} &
\multicolumn{1}{|l|}{%
\begin{tabular}{c}
{\small [ 5, 18 ]}%
\end{tabular}%
} & \multicolumn{1}{|l|}{%
\begin{tabular}{c}
{\small [ 0, 18 ]}%
\end{tabular}%
} \\ \hline
{\small [\ 28, 2 ]} & {\small C28} & {\small 4} & {\small 4} & {\small 4} &
\multicolumn{1}{|l|}{%
\begin{tabular}{c}
{\small [\ 5, 2 ], [ 6, 2 ]}%
\end{tabular}%
} & \multicolumn{1}{|l|}{%
\begin{tabular}{c}
{\small [\ 0, 2 ], [ 1, 1 ], [ 2, 1 ]}%
\end{tabular}%
} \\ \hline
{\small [\ 28, 3 ]} & {\small D28} & {\small 256} & {\small 31} & {\small 7} &
\multicolumn{1}{|l|}{%
\begin{tabular}{l}
{\small [ 4, 3 ], [ 5, 18 ], [ 6, 1 ], } \\
{\small [ 7, 9 ]}%
\end{tabular}%
} & \multicolumn{1}{|l|}{%
\begin{tabular}{c}
{\small [ 0, 31 ]}%
\end{tabular}%
} \\ \hline
{\small [\ 28, 4 ]} & {\small C14xC2} & {\small 40} & {\small 9} & {\small 4} &
\multicolumn{1}{|l|}{%
\begin{tabular}{l}
{\small [ 4, 3 ], [ 5, 2 ], [ 6, 2 ], } \\
{\small [ 7, 1 ], [ 10, 1 ]}%
\end{tabular}%
} & \multicolumn{1}{|l|}{%
\begin{tabular}{c}
{\small [ 0, 7 ], [ 1, 1 ], [ 2, 1 ]}%
\end{tabular}%
} \\ \hline
{\small [\ 29, 1 ]} & {\small C29} & {\small 1} & {\small 1} & {\small 1} &
\multicolumn{1}{|l|}{%
\begin{tabular}{c}
{\small [ 2, 1 ]}%
\end{tabular}%
} & \multicolumn{1}{|l|}{%
\begin{tabular}{c}
{\small [ 1, 1 ]}%
\end{tabular}%
} \\ \hline
{\small [\ 30, 1 ]} & {\small C5xS3} & {\small 20} & {\small 10} & {\small 6} &
\multicolumn{1}{|l|}{%
\begin{tabular}{c}
{\small [ 5, 5 ], [ 6, 5 ]}%
\end{tabular}%
} & \multicolumn{1}{|l|}{%
\begin{tabular}{c}
{\small [ 0, 10 ]}%
\end{tabular}%
} \\ \hline
{\small [\ 30, 2 ]} & {\small C3xD10} & {\small 52} & {\small 14} & {\small 6} &
\multicolumn{1}{|l|}{%
\begin{tabular}{c}
{\small [ 5, 7 ], [ 6, 7 ]}%
\end{tabular}%
} & \multicolumn{1}{|l|}{%
\begin{tabular}{c}
{\small [ 0, 14 ]}%
\end{tabular}%
} \\ \hline
{\small [\ 30, 3 ]} & {\small D30} & {\small 260} & {\small 35} & {\small 9} &
\multicolumn{1}{|l|}{%
\begin{tabular}{c}
{\small [ 5, 35 ]}%
\end{tabular}%
} & \multicolumn{1}{|l|}{%
\begin{tabular}{c}
{\small [ 0, 35 ]}%
\end{tabular}%
} \\ \hline
{\small [\ 30, 4 ]} & {\small C30} & {\small 4} & {\small 4} & {\small 4} &
\multicolumn{1}{|l|}{%
\begin{tabular}{c}
{\small [ 5, 1 ], [ 6, 2 ], [ 8, 1 ]}%
\end{tabular}%
} & \multicolumn{1}{|l|}{%
\begin{tabular}{c}
{\small [ 0, 3 ], [ 1, 1 ]}%
\end{tabular}%
} \\ \hline
{\small [\ 31, 1 ]} & {\small C31} & {\small 1} & {\small 1} & {\small 1} &
\multicolumn{1}{|l|}{%
\begin{tabular}{c}
{\small [ 2, 1 ]}%
\end{tabular}%
} & \multicolumn{1}{|l|}{%
\begin{tabular}{c}
{\small [ 1, 1 ]}%
\end{tabular}%
} \\ \hline
\end{longtable}


\begin{thebibliography}{99}

\bibitem{tamar1} \textsc{Datuashvili, T.}, Central Series for Groups with Action and Leibniz Algebras. \emph{Georgian Mathematical Journal},
 Volume 9, Number 4, 671-682, (2002).

\bibitem{tamar2} \textsc{Datuashvili, T.}, Witt's theorem for groups with action and free Leibniz algebras. \emph{Georgian Mathematical Journal},
 Volume 11, Number 4, 691-712, (2004).

\bibitem{gap} \textsc{GAP - Groups, Algortihms, and Programming, Version 4}, Lehrstuhl D f\"{u}r Mathematik, RWTH Aachen Germany and School of
Mathematical and Computational Sciences, \emph{U. St. Andrews, Scotland}, (1997).

\bibitem{hug1} \textsc{Hug, S. A.}, Commutator, nilpotency and solvability in categories. \emph{Quart. J. Math. Oxford Series},
 Volume 19, Number 1, 363-389, (1968).

\bibitem {loday} \textsc{Loday, J.-L.}, Une version non commutative des alg\`{e}bres de Lie: les alg\`{e}bres de Leibniz. emph{Enseign. Math.} 39, 2, 269--293 (1993).

\bibitem {loday1} \textsc{Loday, J.-L.}, Algebraic K-theory and the conjectural Leibniz K-theory, Special issue
in honor of Hyman Bass on his seventieth birthday. Part II. K-Theory 30 (2003), No.
2, 105-127.

\bibitem {Orz} \textsc{Orzech, G.}, Obstruction theory in algebraic categories I and II, \emph{J. Pure Appl. Algebra}, 2, 287-314 and 315-340 (1972).

\bibitem{ikion} \textsc{Witt, E.}, Treue Darstellung Liescher Ringe. \emph{J. Reine Angew. Math. 1}, 152-160, (1967).


\end{thebibliography}
\end{document}